\documentclass[11pt,reqno]{amsart}
\usepackage[T1]{fontenc}
\usepackage[utf8]{inputenc}
\usepackage{lmodern}
\usepackage{microtype}
\usepackage{amsmath,amssymb,mathtools}
\usepackage{mathrsfs}
\usepackage{enumitem}
\usepackage{xcolor}
\usepackage{tikz}
\usepackage{xurl}
\usepackage[numbers,sort&compress]{natbib}
\usepackage[pdfusetitle,colorlinks=true,
  pdftitle={A uniform commutator bound in finite von Neumann algebras},
  pdfauthor={Jiaqi Wang},linkcolor=blue!45!black,citecolor=blue!45!black,
  urlcolor=blue!45!black]{hyperref}

\newtheorem{theorem}{Theorem}[section]
\newtheorem{proposition}[theorem]{Proposition}
\newtheorem{lemma}[theorem]{Lemma}
\newtheorem{corollary}[theorem]{Corollary}

\theoremstyle{definition}
\newtheorem{definition}[theorem]{Definition}
\theoremstyle{remark}
\newtheorem{remark}[theorem]{Remark}

\newcommand{\M}{\mathcal M}
\newcommand{\N}{\mathcal N}
\newcommand{\C}{\mathbb C}

\newcommand{\ii}{\mathrm i}
\newcommand{\Rea}{\operatorname{Re}}
\newcommand{\Ima}{\operatorname{Im}}
\newcommand{\norm}[1]{\left\lVert#1\right\rVert}
\newcommand{\comm}[2]{[#1,#2]}
\newcommand{\lsupp}{\ell}
\newcommand{\rsupp}{r}
\newcommand{\chii}{\chi}
\newcommand{\lev}{\operatorname{lev}}
\newcommand{\growthF}{F}
\numberwithin{equation}{section}
\allowdisplaybreaks[1]
\title[Uniform commutator bounds in finite von Neumann algebras]
{A uniform commutator bound in finite von Neumann algebras}
\author[Jiaqi Wang]{Jiaqi Wang}
\address{State Key Laboratory of Mathematical Sciences,
Academy of Mathematics and Systems Science,
Chinese Academy of Sciences, Beijing 100190, China}
\address{University of Chinese Academy of Sciences,
Beijing 100049, China}
\email{jiaqiwang@amss.ac.cn}
\date{}

\begin{document}

\begin{abstract}
Let $\M$ be a finite von Neumann algebra with normalized
center-valued trace $T_\M$. We prove that every $A\in\M$ with
$T_\M(A)=0$ admits a representation $A=BC-CB$, where
$B,C\in\M$ satisfy $\norm B\norm C\le K\norm A$ for an absolute
constant $K$ independent of $\M$.
\end{abstract}

\maketitle

\section{Introduction}\label{sec:intro}

The commutator problem has a long history in both matrix and
operator algebras. Shoda \cite{Shoda} proved that a complex matrix
is a commutator precisely when its trace is zero.
Johnson, Ozawa and Schechtman \cite{JOS} asked whether these
representations admit a norm bound independent of the matrix
size, and obtained bounds $K_\varepsilon n^\varepsilon$ for every
$\varepsilon>0$. A dimension-independent bound was established in
\cite{DIC}; that result supplies the matrix case used here.

For finite von Neumann algebras, Fack and de la Harpe \cite{FdH}
identified the kernel of the center-valued trace with the space
of finite sums of commutators. In type $\mathrm{II}_1$ factors,
Marcoux \cite{Marcoux} obtained a quantitative decomposition into
two commutators, and Wen, Fang and Yao \cite[Corollary~4.1]{WFY}
proved that every trace-zero element is a single commutator.
The latter result is an elementary consequence of their
simultaneous compression theorem, recalled in Section~\ref{sec:prelim}.

We establish a uniform norm bound for general finite von Neumann
algebras. Let $\M$ be such an algebra, with normalized center-valued
trace $T_\M:\M\to Z(\M)$. Every additive commutator
$\comm BC=BC-CB$ belongs to $\ker T_\M$.
Our main theorem shows that this necessary condition is sufficient,
with a norm bound independent of $\M$.

\begin{theorem}\label{thm:main}
There is an absolute constant $K>0$ such that, for every finite
von Neumann algebra $\M$ with normalized center-valued trace
$T_\M$ and every $A\in\M$ with $T_\M(A)=0$, there exist
$B,C\in\M$ satisfying $A=\comm BC$ and
\begin{equation}\label{eq:finitebound}
\norm B\norm C\le K\norm A.
\end{equation}
No separability or $\sigma$-finiteness assumption is required.
\end{theorem}

Since $T_\M(BC-CB)=0$, vanishing of the center-valued trace
is therefore equivalent to being a single additive commutator.

The main contribution of this paper is a uniform norm bound for single
commutator representations in type $\mathrm{II}_1$ factors.
The simultaneous compression theorem provides exact scalar
compressions but no uniform bound on the number of blocks.
We obtain the following result by extracting a corner of prescribed
trace and by paving elements of small normalized $L^2$-norm.

\begin{theorem}
\label{thm:factor}
There is an absolute constant $K'>0$ such that, for every type
$\mathrm{II}_1$ factor $(\M,\tau)$ and every $A\in\M$ with
$\tau(A)=0$, there are $B,C\in\M$ satisfying
\[
A=\comm BC,\qquad\norm B\norm C\le K'\norm A.
\]
\end{theorem}

If $K_{\mathrm I}$ is a constant from the matrix theorem of
\cite{DIC}, the main theorem holds with
\begin{equation}\label{eq:finiteconstant}
K=\max\{K_{\mathrm I},K'\}.
\end{equation}
Thus the passage from finite factors to arbitrary finite
von Neumann algebras entails no further loss in the constant.

The constant can be chosen to satisfy
\[
K\le 2^{\,2^{\,2^{46}}}.
\]
This bound is not optimized; its derivation is given in
Section~\ref{sec:conclusion}.

Section~\ref{sec:prelim} collects notation and the results of
Wen--Fang--Yao and Marcoux used below. To prove
Theorem~\ref{thm:factor}, we split according to
$\norm A_2/\norm A$. In the high-$L^2$ case,
Section~\ref{sec:riccati} extracts an invertible off-diagonal corner
of controlled trace and uses a similarity and an identity-corner
commutator construction to obtain a norm bound.
For the low-$L^2$ case, Section~\ref{sec:mass} constructs
equal-trace compressions with small operator norm and exactly
zero trace. Section~\ref{sec:tree} assembles the iterated paving
using Sylvester equations; geometric decay ensures norm convergence,
including along infinite branches. This gives
$K'=3\cdot2^{18}+K_H(2^{-19})$, with $K_H$ defined in \eqref{eq:KH}.

Section~\ref{sec:finite} passes to arbitrary finite von Neumann
algebras. For algebras with separable predual, we use measurable
selection to combine uniformly bounded commutator factors in a
central decomposition. We then extend the result to algebras with
a faithful normal tracial state by passing to a subalgebra with
separable predual in which the given element still has zero
center-valued trace. Finally, decomposition into orthogonal central
summands gives the general case.

\section{Preliminaries}\label{sec:prelim}
\subsection{Notation and traces}
Throughout, $\norm x$ denotes the operator norm.
{Except in the assembly argument of Section~\ref{sec:tree},
which is formulated for an arbitrary unital $C^*$-algebra,
Sections~\ref{sec:riccati}--\ref{sec:tree} concern a type
$\mathrm{II}_1$ factor $(\M,\tau)$ with its normalized faithful
normal trace, and we set $\norm x_2=\tau(x^*x)^{1/2}$.}
We write
$\Rea x=(x+x^*)/2$, $\Ima x=(x-x^*)/(2\ii)$, and
$\lsupp(x),\rsupp(x)$ for the left and right support projections.
For projections, $p\sim q$ denotes Murray--von Neumann equivalence.
A partition of unity consists of finitely many orthogonal
projections summing to $1$. For $p\ne0$, the corner $p\M p$ carries
its normalized trace $\tau_p=\tau(p)^{-1}\tau|_{p\M p}$; all
$L^2$-norms in a corner are taken with respect to this trace unless
otherwise stated. We use projection comparison, diffuseness, polar
decomposition and the standard norm and trace estimates without
further comment; see \cite[Chapters 2--4]{AP}.

For an equal-trace partition $e_1,\ldots,e_R$, choose partial
isometries $v_j$ with $v_j^*v_j=e_1$, $v_jv_j^*=e_j$ and $v_1=e_1$.
The standard isometric $*$-isomorphism onto matrices over
$\N=e_1\M e_1$ is
\begin{equation}\label{eq:unitsR}
\Phi(x)=(v_j^*xv_k)_{j,k=1}^R,\qquad
\Phi^{-1}((y_{jk}))=\sum_{j,k}v_jy_{jk}v_k^*.
\end{equation}
With $\tau_\N=R\tau|_\N$, its trace formula is
\begin{equation}\label{eq:traceR}
\tau(x)=\frac1R\sum_j\tau_\N(\Phi(x)_{jj}),\qquad
\tau(xy)=\frac1R\sum_{j,k}\tau_\N(\Phi(x)_{jk}\Phi(y)_{kj}).
\end{equation}

For a finite von Neumann algebra $\M$ with separable predual
and a faithful normal tracial state $\tau$, we write
\[
(\M,\tau)=\int_X^\oplus(\M_x,\tau_x)\,d\mu(x)
\]
for its central direct-integral decomposition over a standard
probability space $(X,\mu)$. Here $\M_x$ is a finite factor with
normalized trace $\tau_x$ for almost every $x$, and
$Z(\M)$ is identified with $L^\infty(X,\mu)$.
An element $A\in\M$ is represented by an essentially bounded
measurable field $x\mapsto A_x\in\M_x$, determined up to equality
almost everywhere, and is written $A=\int_X^\oplus A_x\,d\mu(x)$.
Algebraic operations are performed fiberwise, and
\[
\norm A=\mathop{\mathrm{ess\,sup}}_{x\in X}\norm{A_x},
\qquad
\tau(A)=\int_X\tau_x(A_x)\,d\mu(x).
\]
Thus the notation for the pair $(\M,\tau)$ records both the
algebra decomposition and the disintegration of the trace;
see \cite[Chapter~IV, Section~8]{Takesaki}.

\subsection{Simultaneous compression and single commutators}

Cao, Fang and Yao \cite[Theorem~1.1]{CFY} established the
finite scalar-compression theorem for a single self-adjoint element
of a type $\mathrm{II}_1$ factor. Wen, Fang and Yao subsequently
extended this to an arbitrary finite family of elements. We recall
their exact simultaneous compression result
\cite[Theorem~3.1]{WFY}, used both in block extraction
and in paving.

\begin{theorem}
\label{thm:WFY}
Let $(\M,\tau)$ be a type $\mathrm{II}_1$ factor and
$X_1,\dots,X_n\in\M$.  Then there is a finite family of mutually
orthogonal nonzero projections $E_1,\dots,E_N\in\M$ with
$\sum_{i=1}^NE_i=1$ and
\[
E_iX_jE_i=\tau(X_j)\,E_i
\qquad(1\le i\le N,\ 1\le j\le n).
\]
\end{theorem}

\begin{corollary}[Wen--Fang--Yao {\cite[Corollary~4.1]{WFY}}]
\label{cor:WFYcomm}
An element $A$ of a type $\mathrm{II}_1$ factor $(\M,\tau)$ is a
single commutator if and only if $\tau(A)=0$.
\end{corollary}

Corollary~\ref{cor:WFYcomm} asserts existence in each factor;
the uniform norm bound needed here is an additional quantitative
conclusion.

\subsection{Marcoux's two-commutator estimate}\label{subsec:marcoux}

Marcoux's theorem \cite[Corollary~3.11 and Remark~5.3]{Marcoux},
after rescaling the factors, gives the following input: every
trace-zero $H$ in a type $\mathrm{II}_1$ factor is of the form
\[
H=\comm{U_1}{V_1}+\comm{U_2}{V_2},\qquad
\norm{U_i}\le1,\qquad\norm{V_i}\le\Lambda_0\norm H,
\qquad\Lambda_0=596100.
\]
For $H\ne0$, write $H/\norm H=\comm dz+\comm sv$ with
Marcoux's bounds $\norm d\le100$, $\norm z\le5961$,
$\norm s\le356928$ and $\norm v\le1$. Taking
\[
U_1=d/100,\quad V_1=100\norm H\,z,\qquad
U_2=v,\quad V_2=-\norm H\,s
\]
gives the stated bounds, using $\comm sv=\comm v{-s}$.
For $H=0$, take all four factors to be zero.

\section{The high-\texorpdfstring{$L^2$}{L2} case}\label{sec:riccati}

We first obtain a uniform bound for elements with a lower bound on
$\norm A_2/\norm A$.

\begin{theorem}\label{thm:high}
For every $0<t\le1$ there is a constant $K_H(t)$ such that, in every
type $\mathrm{II}_1$ factor $(\M,\tau)$, every $A\in\M$ with
\[
\tau(A)=0,\qquad \norm A_2\ge t\,\norm A
\]
satisfies $A=\comm BC$ with $\norm B\norm C\le K_H(t)\norm A$,
with an explicit choice given in \eqref{eq:KH}.
\end{theorem}

\subsection{Commutators with an identity corner}\label{subsec:core}

We shall use two elementary estimates, valid in any unital
$C^*$-algebra; compare \cite[Lemmas~2.4--2.5]{DIC}.
For projections $p,q$, a right-hand side $Y\in p\M q$, elements $U\in p\M p$, $V\in q\M q$, and
$z,w\in\C$ with $\norm U+\norm V<|z-w|$, the Sylvester equation
$(z-w)X+UX-XV=Y$ has a solution in $p\M q$ satisfying
\begin{equation}\label{eq:sylvbound}
\norm X\le\frac{\norm Y}{|z-w|-\norm U-\norm V}.
\end{equation}
This follows by a Neumann series on $p\M q$.
For a hollow block operator $T=(T_{ij})_{i,j\le b}$,
Cauchy--Schwarz applied to its rows gives
\begin{equation}\label{eq:hollowbound}
T_{ii}=0,\quad\norm{T_{ij}}\le c\ (i\ne j)
\quad\Longrightarrow\quad\norm T\le(b-1)c.
\end{equation}

The next proposition represents a two-by-two operator with an
identity off-diagonal block as a single commutator. It extends
\cite[Lemma~2.3]{DIC} to a unital $C^*$-algebra, with a separate
parameter for the two-commutator input. Only contractivity of the
first factors is needed, allowing us to use Marcoux's theorem.

\begin{proposition}
\label{prop:matrixcomm}
Let $\N$ be a unital C$^*$-algebra, $x\ge1$, $\lambda\ge1$, and
\[
M'=\begin{pmatrix}A&1\\F&D\end{pmatrix}\in M_2(\N),\qquad
\norm{M'}\le x .
\]
Suppose
\begin{equation}\label{eq:twoinput}
\begin{gathered}
A+D=\comm UV+\comm ZT,\qquad U,Z\in\N,\\
\norm U,\norm Z\le1,\qquad
\norm V,\norm T\le\lambda .
\end{gathered}
\end{equation}
Then $M'=\comm{B_0}{C_0}$ in $M_2(\N)$ with
\[
\norm{B_0}\norm{C_0}\ \le\
\Pi(x,\lambda):=(2\lambda+x+2)^4\,(\sigma+3)\,
\bigl((\sigma+5)\lambda+2\bigr),
\]
where
$\sigma=\sigma(x,\lambda)
:=2\bigl[(2\lambda+x+1)^2+(2x+1)(2\lambda+x+1)
+6\lambda+2x+1\bigr]$.
\end{proposition}

\begin{proof}
Put $\rho_0=2\lambda+x+1$ and
\[
\widetilde U=\sigma1+U,\qquad D_0=\comm UV-A,\qquad
D_1=\comm ZV+D_0U-Z+\comm U{TZ}-F.
\]
Then $\norm{D_0}\le2\lambda+x$ and
$\norm{D_1}\le6\lambda+2x+1$.
For $Q\in\N$, set
\[
\begin{gathered}
B_c=\begin{pmatrix}\widetilde U&1\\Z&0\end{pmatrix},\qquad
S=\begin{pmatrix}1&0\\Q&1\end{pmatrix},\\[4pt]
C_c=\begin{pmatrix}V&T\\Q-D_0+TZ&V+1-\widetilde UT\end{pmatrix}.
\end{gathered}
\]
Block multiplication and \eqref{eq:twoinput} give
\begin{align*}
\comm{B_c}{C_c}&=\begin{pmatrix}A+Q&1\\
\sigma D_0+D_1+F-Q\widetilde U&D-Q\end{pmatrix},\\
S^{-1}M'S&=\begin{pmatrix}A+Q&1\\F+DQ-QA-Q^2&D-Q\end{pmatrix}.
\end{align*}
Thus these matrices agree precisely when
\begin{equation}\label{eq:riccati}
\sigma(Q-D_0)=Q^2-DQ+Q(A-U)+D_1.
\end{equation}
On the ball $\norm{Q-D_0}\le1$, we have $\norm Q\le\rho_0$.
The map
$\Psi(Q)=D_0+\sigma^{-1}(Q^2-DQ+Q(A-U)+D_1)$ satisfies
\[
\norm{\Psi(Q)-D_0}
\le\frac{\rho_0^2+(2x+1)\rho_0+6\lambda+2x+1}{\sigma}
=\frac12,
\]
and, for $Q,Q'$ in the ball,
\[
\norm{\Psi(Q)-\Psi(Q')}
\le\frac{2\rho_0+2x+1}{\sigma}\norm{Q-Q'}
\le\frac12\norm{Q-Q'}.
\]
The last inequality follows from $\rho_0\ge4$ and the definition
of $\sigma$. The Banach fixed-point theorem gives a solution of
\eqref{eq:riccati}, hence
$M'=\comm{SB_cS^{-1}}{SC_cS^{-1}}$.
Finally,
\[
\norm{B_c}\le\sigma+3,\qquad
\norm{C_c}\le(\sigma+5)\lambda+2,\qquad
\norm S,\norm{S^{-1}}\le2\lambda+x+2.
\]
Conjugating both factors gives the asserted product bound.
\end{proof}

The two-commutator estimate recalled in
Section~\ref{subsec:marcoux} supplies the required input.
Consequently, if $M'$ in Proposition~\ref{prop:matrixcomm} is
trace-zero in $M_2(\N)$ for a type $\mathrm{II}_1$ factor $\N$,
then $\tau_\N(A+D)=0$ and $\norm{A+D}\le2x$, so we may take
$\lambda=2\Lambda_0x$. We shall use the resulting bound
\begin{equation}\label{eq:twox}
M'=\comm{B_0}{C_0},\qquad
\norm{B_0}\norm{C_0}\le\Pi^*(x):=\Pi(x,2\Lambda_0x).
\end{equation}
The normalized trace on $M_2(\N)$ is denoted by $\tau_2$.
If $w^*w=p_2$ and $ww^*=p_1$ for orthogonal equivalent
projections, \eqref{eq:unitsR} in $(p_1+p_2)\M(p_1+p_2)$
with $v_2=w^*$ identifies $w$ with the upper-right identity block
and preserves the trace-zero condition.

\subsection{Shears and absorption}

As in \cite[Eq.~(2.9)]{DIC}, put
\begin{equation}\label{eq:f}
\growthF(x)=x(x+2)^2\bigl(1+x(x+2)^2\bigr)^2\qquad(x\ge1).
\end{equation}
This function is nondecreasing and satisfies $\growthF(x)\ge x$.
The following proposition absorbs the remaining blocks into the
identity-corner representation. Its bound depends only on the
operator norm and the number of blocks.

\begin{proposition}\label{prop:absorb}
Let $(p_i)_{i\le b}$ be a partition of unity in $\M$, $b\ge2$, with
\[
p_1\ne0,\qquad \tau(p_1)=\tau(p_2),\qquad
\tau(p_j)\le\tau(p_1)\quad(3\le j\le b),
\]
and let $w\in\M$ be a partial isometry with $w^*w=p_2$,
$ww^*=p_1$.  Let $A\in\M$ satisfy
\[
\tau(A)=0,\qquad \norm A\le a\ \ (a\ge1),\qquad p_1Ap_2=w .
\]
Then $A=\comm BC$ with
\begin{equation}\label{eq:absorbbound}
\norm B\norm C\ \le\
\Bigl(\frac{\widehat a}{a}\Bigr)^2(b-1)
\Bigl(4\Pi^*(\widehat a)+2(b-1)\widehat a\Bigr),
\qquad \widehat a:=\growthF^{\circ(b-2)}(a),
\end{equation}
where $\growthF^{\circ m}$ is the $m$-fold composition and
$\growthF^{\circ0}=\mathrm{id}$.
\end{proposition}

\begin{proof}
We first claim that there is an invertible $S\in\M$
with
\begin{equation}\label{eq:aftershear}
\begin{aligned}
\widehat A&:=S^{-1}AS,\qquad
\norm{\widehat A}\le\widehat a,\qquad
\chii(S):=\norm S\norm{S^{-1}}\le\widehat a/a,\\
p_1\widehat Ap_2&=w,\qquad
p_j\widehat Ap_j=0\quad(3\le j\le b).
\end{aligned}
\end{equation}

If $b=2$, take $S=1$. Otherwise eliminate the diagonal blocks
$j=3,\ldots,b$ successively, as in \cite[Proposition~2.6]{DIC}.
Suppose the current operator $A^{(0)}$ has norm at most $x\ge a$
and retains $p_1A^{(0)}p_2=w$. By projection comparison choose
$v_j\in p_2\M p_j$ with $v_j^*v_j=p_j$, and set $u_j=wv_j$.
Thus $u_j^*u_j=p_j$.

First put
\[
X=v_j-w^*p_1A^{(0)}p_j,\qquad S_1=1+X,\qquad
A'=S_1^{-1}A^{(0)}S_1.
\]
Since $X\in p_2\M p_j$, $X^2=0$ and $S_1^{-1}=1-X$.
The new $(1,j)$ block is $p_1A'p_j=u_j$, and only diagonal
blocks $2$ and $j$ can change. Moreover,
\[
\chii(S_1)\le(x+2)^2,\qquad\norm{A'}\le x(x+2)^2,
\]
where $\chii(S)=\norm S\norm{S^{-1}}$.
Next put
\[
Y=(p_jA'p_j)u_j^*,\qquad S_2=1+Y,\qquad A''=S_2^{-1}A'S_2.
\]
Here $Y\in p_j\M p_1$, so $Y^2=0$ and $S_2^{-1}=1-Y$.
Only diagonal blocks $1$ and $j$ can change, and
\[
p_jA''p_j=p_jA'p_j-Y(p_1A'p_j)
=p_jA'p_j-(p_jA'p_j)u_j^*u_j=0.
\]
Both shears preserve the $(1,2)$ block and every previously
eliminated diagonal block. Their bounds are
\[
\chii(S_2)\le\bigl(1+x(x+2)^2\bigr)^2,\qquad
\norm{A''}\le \growthF(x).
\]

Thus one elimination step replaces the norm bound $x$ by $\growthF(x)$
with accumulated condition number at most
$(x+2)^2(1+x(x+2)^2)^2=\growthF(x)/x$.  Performing the steps for
$j=3,\dots,b$ and using that $\growthF$ is nondecreasing with
$\growthF(x)\ge x$, the total condition number telescopes:
\[
\chii(S)\le\prod_{m=0}^{b-3}
\frac{\growthF\bigl(\growthF^{\circ m}(a)\bigr)}{\growthF^{\circ m}(a)}
=\frac{\growthF^{\circ(b-2)}(a)}{a}=\frac{\widehat a}{a},
\]
and $\norm{\widehat A}\le\widehat a$, proving
\eqref{eq:aftershear}.

Let $q_1=p_1+p_2$ and $M'=q_1\widehat Aq_1$. Since similarities
preserve the trace and $p_j\widehat Ap_j=0$ for $j\ge3$,
$\tau(M')=\tau(\widehat A)=0$. Also $\norm{M'}\le\widehat a$
and $p_1M'p_2=w$. Under \eqref{eq:unitsR} with $R=2$ in
$q_1\M q_1$, this is a trace-zero matrix with upper-right identity
block. By \eqref{eq:twox}, there are $B_0,C_0\in q_1\M q_1$
with $M'=\comm{B_0}{C_0}$ and
$\norm{B_0}\norm{C_0}\le\Pi^*(\widehat a)$.
Since $M'\ne0$, we may rescale so that
\[
\norm{B_0}=\tfrac14,\qquad \norm{C_0}\le4\Pi^*(\widehat a).
\]
Relabel the
partition as $q_1=p_1+p_2$ and $q_i=p_{i+1}$ for $2\le i\le
k:=b-1$.  Define
\[
\widehat B=B_0+\sum_{i=2}^{k}(i-1)\,q_i,\qquad
\widehat C_{11}=C_0,\qquad \widehat C_{ii}=0\ (2\le i\le k),
\]
and for $i\ne j$ let $\widehat C_{ij}\in q_i\M q_j$ solve
\[
(z_i-z_j)\widehat C_{ij}
+\widehat B^{\,\prime}_{ii}\widehat C_{ij}
-\widehat C_{ij}\widehat B^{\,\prime}_{jj}
=q_i\widehat Aq_j,
\]
where $z_i=i-1$ and $\widehat B'_{11}=B_0$,
$\widehat B'_{ii}=0$ for $i\ge2$, so that
$q_i\widehat Bq_i=z_iq_i+\widehat B'_{ii}$.  The centers are
distinct integers, $|z_i-z_j|\ge1$, and
$\norm{\widehat B'_{ii}}+\norm{\widehat B'_{jj}}\le\tfrac14$, so
the Sylvester estimate \eqref{eq:sylvbound} gives solutions with
\[
\norm{\widehat C_{ij}}
\le\frac{\norm{q_i\widehat Aq_j}}{1-\tfrac14}
\le\tfrac43\widehat a\le2\widehat a .
\]
Set $\widehat C=\sum_{ij}\widehat C_{ij}$. The diagonal blocks
of $\comm{\widehat B}{\widehat C}$ agree with those of $\widehat A$
by construction, and the off-diagonal blocks agree by the Sylvester
equations. Hence $\comm{\widehat B}{\widehat C}=\widehat A$.
The block diagonal norm formula gives
$\norm{\widehat B}\le\max(\tfrac14,k-1)\le k$, while
\eqref{eq:hollowbound} gives
\[
\norm{\widehat C}\le4\Pi^*(\widehat a)+2\widehat a(k-1)
\le4\Pi^*(\widehat a)+2(b-1)\widehat a .
\]

Finally, set $B=S\widehat BS^{-1}$ and
$C=S\widehat CS^{-1}$.  Then
$\comm BC=S\comm{\widehat B}{\widehat C}S^{-1}
=S\widehat AS^{-1}=A$, and
$\norm B\norm C\le\chii(S)^2\norm{\widehat B}\norm{\widehat C}
\le(\widehat a/a)^2(b-1)\bigl(4\Pi^*(\widehat a)
+2(b-1)\widehat a\bigr)$.
\end{proof}

\subsection{Block extraction and the high-\texorpdfstring{$L^2$}{L2} bound}\label{sec:high}

In the matrix argument, repeated vector selection produces an
invertible off-diagonal block \cite[Proposition~3.6]{DIC}.
Here we instead apply Theorem~\ref{thm:WFY} to $(T,T^*T)$ in a
large corner. A maximal family of such compressions yields a
projection of prescribed trace and a quantitative lower bound
on the off-diagonal block.

\begin{lemma}\label{lem:extract}
Let $A\in\M$ with $\norm A=1$, $\tau(A)=0$, and
$\norm A_2\ge t$ for some $0<t\le1$.  Put
$\sigma_t=t^2/64$.
Then there exist projections $e,f\in\M$ and $X\in\M$ with:
\begin{enumerate}[label=(\alph*)]
\item $e\perp f$, $e\sim f$, $\tau(e)=\tau(f)=\sigma_t$;
\item $\norm{eAe}\le t^2/15$;
\item $X=(1-e)Ae$ satisfies $X^*X\ge\dfrac{t^2}{5}\,e$ and
$f=\lsupp(X)$; in particular $fAe=X$ and
$|X|^{-1}\in e\M e$ with $\norm{|X|^{-1}}\le\sqrt5/t$.
\end{enumerate}
\end{lemma}

\begin{proof}
Consider families $\mathcal F=\{r_k\}$ of nonzero, mutually
orthogonal projections in $\M$ with
$\sum_k\tau(r_k)\le\sigma_t$, subject to the invariants
\begin{enumerate}[label=(\Roman*)]
\item $r_kAr_l=0$ and $r_kA^*Ar_l=0$ whenever $k\ne l$;
\item $r_kAr_k=c_kr_k$ for a scalar $c_k$ with
$|c_k|\le t^2/15$;
\item $r_kA^*Ar_k\ge\dfrac{t^2}{4}\,r_k$.
\end{enumerate}
The empty family qualifies.  The union of an increasing chain of
such families is again such a family: the invariants are
conditions on individual members and pairs, and the trace bound
passes to the union since every finite subfamily of the union lies
in some member of the chain.  By Zorn's lemma there is a maximal
family $\mathcal F=\{r_k\}$. Since the positive traces
$\tau(r_k)$ sum to at most $\sigma_t$, this family is countable.
Its strong sum $e=\sum_kr_k$ is a projection satisfying
$\tau(e)=\sum_k\tau(r_k)\le\sigma_t$.

We first show that $\tau(e)=\sigma_t$.  Suppose
$\tau(e)<\sigma_t$.  Set
\[
q=1-\Bigl(e\vee\lsupp(Ae)\vee\lsupp(A^*e)\vee\lsupp(A^*Ae)\Bigr).
\]
By polar decomposition,
$\tau(\lsupp(Ae))=\tau(\rsupp(Ae))\le\tau(e)$ since
$\rsupp(Ae)\le e$, and likewise for $A^*e$ and $A^*Ae$; by
subadditivity of the trace on joins,
\begin{equation}\label{eq:qtrace}
\tau(1-q)\le4\tau(e)<4\sigma_t=\frac{t^2}{16}\le\frac1{16}.
\end{equation}
In particular $\tau(q)\ge\tfrac{15}{16}>0$ and $\N_q:=q\M q$ is a
$\mathrm{II}_1$ factor with normalized
trace $\tau_q=\tau/\tau(q)$.  Let $T=qAq\in\N_q$.

Since $\tau(A)=0$,
$\tau(qAq)=\tau(Aq)=-\tau(A(1-q))$, and by
the trace estimate,
$|\tau(A(1-q))|\le\norm A\,\tau(1-q)\le t^2/16$.  Hence
\begin{equation}\label{eq:cornertrace}
|\tau_q(T)|\le\frac{t^2/16}{\tau(q)}
\le\frac{t^2/16}{15/16}=\frac{t^2}{15}.
\end{equation}

Since $A-qAq=(1-q)A+qA(1-q)$ and
$\norm{(1-q)A}_2,\norm{A(1-q)}_2\le\sqrt{\tau(1-q)}\norm A$,
\eqref{eq:qtrace} gives

\[
\norm T_2\ge\norm A_2-\norm{A-qAq}_2
\ge t-2\sqrt{\tau(1-q)}
\ge t-2\sqrt{t^2/16}=t-\frac t2=\frac t2,
\]
where $\norm\cdot_2$ is computed in $(\M,\tau)$.
Since $\tau(q)\le1$, the normalized corner trace satisfies
\begin{equation}\label{eq:cornerL2}
\tau_q(T^*T)=\frac{\tau(T^*T)}{\tau(q)}\ge\tau(T^*T)
=\norm T_2^2\ge\frac{t^2}{4}.
\end{equation}

Apply Theorem~\ref{thm:WFY} in the
$\mathrm{II}_1$ factor $\N_q$ to the pair $(T,\,T^*T)$: there is a
nonzero projection $E\le q$ with
\[
ETE=\tau_q(T)\,E,\qquad ET^*TE=\tau_q(T^*T)\,E .
\]
These identities remain valid on every subprojection of $E$. By diffuseness,
choose $0\ne r\le E$ with
$\tau(r)=\min\bigl(\tau(E),\,\sigma_t-\tau(e)\bigr)>0$.

We check that $\mathcal F\cup\{r\}$ satisfies the invariants.
For (I), let $r_k\in\mathcal F$.  Since
$\lsupp(Ar_k)\le\lsupp(Ae)$ and $r\le q\perp\lsupp(Ae)$, we get
$rAr_k=r\,\lsupp(Ar_k)Ar_k=0$; since
$\lsupp(A^*r_k)\le\lsupp(A^*e)\perp r$, we get $rA^*r_k=0$ and
hence $r_kAr=(rA^*r_k)^*=0$; since
$\lsupp(A^*Ar_k)\le\lsupp(A^*Ae)\perp r$, we get $rA^*Ar_k=0$ and
$r_kA^*Ar=(rA^*Ar_k)^*=0$.  For (II):
$rAr=r(qAq)r=rTr=\tau_q(T)r$ and $|\tau_q(T)|\le t^2/15$ by
\eqref{eq:cornertrace}.  For (III): since $r\le q$,
\[
rA^*Ar=rA^*qAr+rA^*(1-q)Ar
\ge rA^*qAr=rT^*Tr=\tau_q(T^*T)\,r\ge\frac{t^2}{4}\,r,
\]
using \eqref{eq:cornerL2} and that
$rA^*(1-q)Ar=\bigl((1-q)Ar\bigr)^*\bigl((1-q)Ar\bigr)\ge0$.
Finally $\sum\tau(r_k)+\tau(r)\le\sigma_t$ by the choice of
$\tau(r)$.  This contradicts maximality of $\mathcal F$, proving
the claim; so $\tau(e)=\sigma_t$.

Finally, let $e_F=\sum_{k\in F}r_k$ for finite $F$.
By invariants (I) and (II), $e_FAe_F=\sum_{k\in F}c_kr_k$.
Since $e_F\to e$ strongly and multiplication is jointly strongly
continuous on bounded sets, $eAe=\sum_kc_kr_k$ in the strong
operator topology. The block diagonal norm formula gives
\[
\norm{eAe}=\sup_k|c_k|\le\frac{t^2}{15},
\]
proving (b).  Similarly $e_FA^*Ae_F=\sum_{k\in F}r_kA^*Ar_k\ge
\tfrac{t^2}4e_F$, and positivity passes to the strong limit:
\begin{equation}\label{eq:eAAe}
eA^*Ae\ \ge\ \frac{t^2}{4}\,e .
\end{equation}
Let $X=(1-e)Ae$. By \eqref{eq:eAAe}, the bound on $\norm{eAe}$,
and $t\le1$, we have
\[
\begin{aligned}
X^*X&=eA^*Ae-(eAe)^*(eAe)\\
&\ge\Bigl(\frac{t^2}{4}-\frac{t^4}{225}\Bigr)e
\ge\frac{t^2}{5}\,e.
\end{aligned}
\]
Thus $|X|$ is invertible in $e\M e$, with
$\norm{|X|^{-1}}\le\sqrt5/t$, and $\rsupp(X)=e$.

Set $f=\lsupp(X)$. Since $X=(1-e)X$, we have $f\le1-e$.
Polar decomposition gives $f\sim e$, hence
$\tau(f)=\tau(e)=\sigma_t$, proving (a).
Finally, $fAe=f(1-e)Ae=fX=X$, which completes (c).
\end{proof}

\begin{remark}
Unlike the matrix argument, no exact neutrality $eAe=0$ is
needed: the small diagonal corner $eAe$ is simply carried along as
one of the arbitrary diagonal blocks tolerated by
Proposition~\ref{prop:absorb}.
\end{remark}

An explicit constant in Theorem~\ref{thm:high} is as follows.
With $L=\lceil64/t^2\rceil$ and $a_1=\sqrt5/t$,
\begin{equation}\label{eq:KH}
K_H(t)=\frac{5}{t^2}\,
\Bigl(\frac{\growthF^{\circ L}(a_1)}{a_1}\Bigr)^{2}(L+1)
\Bigl(4\,\Pi^*\bigl(\growthF^{\circ L}(a_1)\bigr)
+2(L+1)\growthF^{\circ L}(a_1)\Bigr).
\end{equation}
Here $\Pi^*(x)=\Pi(x,2\Lambda_0x)$ with $\Lambda_0=596100$ and
$\Pi$ is the function in Proposition~\ref{prop:matrixcomm}.

\begin{proof}[Proof of Theorem~\ref{thm:high}]
If $A=0$, take $B=C=0$. Otherwise, by homogeneity we may
assume $\norm A=1$, so $\norm A_2\ge t$.  Apply Lemma~\ref{lem:extract} to get
$e,f,X$ with properties (a)--(c), and let $X=w|X|$ be the polar
decomposition: $w\in\M$ is a partial isometry with
$w^*w=\rsupp(X)=e$ and $ww^*=\lsupp(X)=f$.

Set $S_0=(1-e)+|X|^{-1}$, where the inverse is taken in $e\M e$.
Its inverse is $(1-e)+|X|$. By Lemma~\ref{lem:extract} and
$\norm X\le1$, the block diagonal norm formula gives
\[
\chii(S_0)\le\sqrt5/t.
\]
Put $A^{(1)}=S_0^{-1}AS_0$. Then
$\norm{A^{(1)}}\le\sqrt5/t=a_1$ and $\tau(A^{(1)})=0$.
Since $f\perp e$ and $fAe=X=w|X|$, its $(f,e)$ corner is
\[
fS_0^{-1}AS_0e=fAe\,|X|^{-1}=w.
\]

Set $p_1=f$, $p_2=e$; then $p_1A^{(1)}p_2=w$
with $w^*w=p_2$, $ww^*=p_1$, and
$\tau(p_1)=\tau(p_2)=\sigma_t$.  Since $t\le1$,
$\sigma_t\le1/64$, so $g:=1-e-f\ne0$ with
$\tau(g)=1-2\sigma_t$.  By diffuseness, partition $g$
into nonzero projections $p_3,\dots,p_b$ each of trace at most
$\sigma_t$, with
\[
b-2=\Bigl\lceil\frac{1-2\sigma_t}{\sigma_t}\Bigr\rceil
\le\Bigl\lceil\frac1{\sigma_t}\Bigr\rceil
=\Bigl\lceil\frac{64}{t^2}\Bigr\rceil=L .
\]
Thus $b\le L+2$ and
$\tau(p_j)\le\sigma_t=\tau(p_1)$ for $j\ge3$.

Proposition~\ref{prop:absorb} applies to $A^{(1)}$ with $a=a_1$.
Hence
$A^{(1)}=\comm{B_1}{C_1}$ with
\[
\begin{gathered}
\norm{B_1}\norm{C_1}
\le\Bigl(\frac{\widehat a}{a_1}\Bigr)^2(b-1)
\bigl(4\Pi^*(\widehat a)+2(b-1)\widehat a\bigr),\\
\widehat a=\growthF^{\circ(b-2)}(a_1)\le \growthF^{\circ L}(a_1),
\end{gathered}
\]
using that $\growthF$ is nondecreasing with $\growthF(x)\ge x$ and that the
bound is nondecreasing in $\widehat a$ and in $b-1\le L+1$.

Set $B=S_0B_1S_0^{-1}$ and $C=S_0C_1S_0^{-1}$.  Then
$\comm BC=S_0A^{(1)}S_0^{-1}=A$ and
\[
\begin{aligned}
\norm B\norm C
&\le\chii(S_0)^2\norm{B_1}\norm{C_1}\\
&\le\frac5{t^2}\Bigl(\frac{\growthF^{\circ L}(a_1)}{a_1}\Bigr)^2(L+1)\\
&\qquad\times\Bigl(4\Pi^*\bigl(\growthF^{\circ L}(a_1)\bigr)
      +2(L+1)\growthF^{\circ L}(a_1)\Bigr)\\
&=K_H(t).
\end{aligned}
\]
\end{proof}

\section{Exact trace-zero paving}\label{sec:mass}

We now construct equal-trace compressions with small operator norm
and exactly zero trace. We first partition the unit so that the
trace pairings with a prescribed finite family are equally
distributed among the resulting projections.

\begin{lemma}\label{lem:equidist}
Let $\N$ be a $\mathrm{II}_1$ factor, $T_1,\dots,T_n\in\N$, and
$R\ge1$ an integer.  There is a partition of unity
$\mathsf F_1,\dots,\mathsf F_R$ in $\N$ with
\[
\tau_\N(\mathsf F_m)=\frac1R,\qquad
\tau_\N(T_i\mathsf F_m)=\frac{\tau_\N(T_i)}R
\qquad(1\le i\le n,\ 1\le m\le R).
\]
\end{lemma}

\begin{proof}
By Theorem~\ref{thm:WFY} there is a partition of unity
$F_1,\dots,F_M$ in $\N$ with $F_aT_iF_a=\tau_\N(T_i)F_a$ for all
$a,i$.  By diffuseness split each $F_a$ into $R$
mutually orthogonal subprojections $F_{a,1},\dots,F_{a,R}$ with
$\tau_\N(F_{a,m})=\tau_\N(F_a)/R$, and put
$\mathsf F_m=\sum_aF_{a,m}$.  Then $(\mathsf F_m)_m$ is a partition
of unity with $\tau_\N(\mathsf F_m)=\frac1R\sum_a\tau_\N(F_a)=\frac1R$,
and compression to each $F_{a,m}$ gives, by traciality,
\[
\tau_\N(T_i\mathsf F_m)=\sum_a\tau_\N(F_{a,m}T_iF_{a,m})
=\sum_a\tau_\N(T_i)\,\tau_\N(F_{a,m})=\frac{\tau_\N(T_i)}R .
\]
\end{proof}

Combining this equidistribution with a spectral cut and Fourier
projections yields the required paving. The trace of each
compression vanishes exactly, so the construction can be iterated
inside the resulting corners.

\begin{theorem}\label{thm:mass}
Let $R\ge2$ be an integer and $0<s\le1$, and let $0<t\le s/\sqrt{2R}$.  Then
for every $\mathrm{II}_1$ factor $(\M,\tau)$ and
every $X\in\M$ with $\tau(X)=0$ and $\norm X_2<t\norm X$ there is a
partition of unity $p_1,\dots,p_R\in\M$ with
\[
\begin{gathered}
\tau(p_\ell)=\frac1R,\qquad
\tau(p_\ell Xp_\ell)=0,\\
\norm{p_\ell Xp_\ell}\le\rho\norm X,\qquad
\rho=2s+\frac2R\qquad(1\le\ell\le R).
\end{gathered}
\]
\end{theorem}

\begin{proof}
By homogeneity we may assume $\norm X=1$; then $\norm X_2<t$.

Write $X=H+\ii G$ with $H,G$ self-adjoint, and let
$e_H=1_{\{|\lambda|>s\}}(H)$ and
$e_G=1_{\{|\lambda|>s\}}(G)$. Put
\[
e=e_H\vee e_G,\qquad X_1=He_H+\ii Ge_G,\qquad X_0=X-X_1.
\]
The spectral theorem and $\norm H_2,\norm G_2\le\norm X_2$ give
\[
\begin{gathered}
\norm{X_0}\le2s,\qquad X_1=eX_1e,\qquad\norm{X_1}\le2,\\
\tau(e)\le\frac{\norm H_2^2+\norm G_2^2}{s^2}
\le\frac{2\norm X_2^2}{s^2}<\frac1R.
\end{gathered}
\]
The last inequality uses $\norm X_2<t\le s/\sqrt{2R}$.

By diffuseness
choose a projection $q\le1-e$ with $\tau(q)=\frac1R-\tau(e)$ and set
$e_1=e+q$, so that $e\le e_1$ and $\tau(e_1)=\frac1R$.  Partition
$1-e_1$ into projections $e_2,\dots,e_R$ of trace $\frac1R$ each.  Choose $v_j$, $\N=e_1\M e_1$, $\tau_\N$ and
$\Phi:\M\to M_R(\N)$ as in the identification \eqref{eq:unitsR}, and let
$T=\Phi(X)$, with entries $T_{jk}=v_j^*Xv_k\in\N$.  Note that
\eqref{eq:traceR} gives
\begin{equation}\label{eq:diagsum}
\sum_{j=1}^R\tau_\N(T_{jj})=R\,\tau(X)=0 .
\end{equation}

Apply
Lemma~\ref{lem:equidist} in the $\mathrm{II}_1$ factor $\N$ to the
finite family $\{T_{jk}\}_{j,k=1}^R$: there is a partition of unity
$\mathsf F_1,\dots,\mathsf F_R$ in $\N$ with
$\tau_\N(\mathsf F_m)=\frac1R$ and
\begin{equation}\label{eq:equi}
\tau_\N(T_{jk}\mathsf F_m)=\frac{\tau_\N(T_{jk})}{R}
\qquad(1\le j,k,m\le R).
\end{equation}

For $\ell\in\mathbb Z/R$, let $P_\ell$ be the rank-one
projection onto the Fourier vector
$R^{-1/2}(\omega^{\ell j})_{j=1}^R$, where $\omega=e^{2\pi\ii/R}$.
These projections form a partition of unity in $M_R(\C)$ and satisfy
\begin{equation}\label{eq:fourieravg}
(P_\ell)_{jj}=\frac1R,\qquad
\frac1R\sum_{m\in\mathbb Z/R}(P_{\ell+m})_{jk}
=\frac1R\delta_{jk}.
\end{equation}

Define in $M_R(\N)$
\[
\widehat p_\ell
=\sum_{m=1}^{R}P_{\ell+m}\otimes\mathsf F_m
:=\Bigl(\sum_{m=1}^R(P_{\ell+m})_{jk}\mathsf F_m\Bigr)_{j,k=1}^R,
\qquad
p_\ell=\Phi^{-1}(\widehat p_\ell)
\]
for $\ell=1,\dots,R$. Orthogonality of the Fourier projections
and of the $\mathsf F_m$ gives
\[
\widehat p_\ell^*=\widehat p_\ell,\qquad
\widehat p_\ell\widehat p_k=\delta_{\ell k}\widehat p_\ell,
\qquad \sum_\ell\widehat p_\ell=1.
\]
Thus $(p_\ell)$ is a partition of unity. Moreover,
$(\widehat p_\ell)_{jj}=R^{-1}1_\N$, so
\eqref{eq:traceR} gives $\tau(p_\ell)=1/R$.

Since $\Phi(e_1)$ is the
matrix unit with $e_1$ in the $(1,1)$ slot,
\[
\begin{aligned}
\Phi(e_1p_\ell e_1)&=\Phi(e_1)\widehat p_\ell\Phi(e_1)\\
&=\Bigl(\sum_m(P_{\ell+m})_{11}\mathsf F_m\Bigr)\oplus0\\
&=\frac1R\Bigl(\sum_m\mathsf F_m\Bigr)\oplus0=\frac1R\Phi(e_1),
\end{aligned}
\]
i.e.\ $e_1p_\ell e_1=\frac1Re_1$.  Hence
$\norm{ep_\ell}^2\le\norm{e_1p_\ell}^2=\norm{e_1p_\ell e_1}=\frac1R$,
using $e\le e_1$.

Since $X_1=eX_1e$,
\[
\begin{aligned}
p_\ell Xp_\ell&=p_\ell X_0p_\ell+(p_\ell e)X_1(ep_\ell),\\
\norm{p_\ell Xp_\ell}&\le\norm{X_0}+\norm{p_\ell e}\,\norm{X_1}\,
\norm{ep_\ell}\\
&\le2s+\frac2R=\rho .
\end{aligned}
\]

By \eqref{eq:traceR} with $x=X$,
$y=p_\ell$, then \eqref{eq:equi} and \eqref{eq:fourieravg},
\begin{align*}
\tau(p_\ell Xp_\ell)&=\tau(Xp_\ell)\\
&=\frac1R\sum_{j,k}\tau_\N\Bigl(T_{jk}\sum_m(P_{\ell+m})_{kj}
\mathsf F_m\Bigr)\\
&=\frac1R\sum_{j,k}\sum_m(P_{\ell+m})_{kj}\,
\frac{\tau_\N(T_{jk})}R\\
&=\frac1R\sum_{j,k}\tau_\N(T_{jk})\cdot\frac1R\,\delta_{jk}\\
&=\frac1{R^2}\sum_j\tau_\N(T_{jj})=0
\end{align*}
by \eqref{eq:diagsum}.  This completes the proof.
\end{proof}

\section{Hierarchical assembly and the factor bound}
\label{sec:tree}

We assemble commutator representations on the terminal corners of
a tree of projections. This construction works in any unital
$C^*$-algebra $\M$: the estimates below ensure norm convergence of
both factors, even when the tree has infinite branches.
We record the decay and representation bounds required of the tree.

\begin{definition}\label{def:tree}
Let $r\ge2$ be an integer. An \emph{assembly tree with branching
bound $r$} is obtained from $1\in\M$ by recursively splitting
selected projections as
\[
p=c_1+\cdots+c_{m(p)},\qquad 2\le m(p)\le r,
\]
where the $c_i\in\M$ are nonzero mutually orthogonal projections,
called the children of $p$. We require $1$ to be split and allow
infinitely many levels. Projections that are split are called
\emph{internal nodes}; those left unsplit are \emph{leaves}.

The root $1$ has level zero, and children have level one greater
than their parent. Write $\lev(p)$ for the level of $p$, and
$N_j,L_j$ for the internal nodes and leaves at level $j$.

Let $A\in\M$, $0<\rho<1$ and $\kappa\ge0$. The tree, together
with chosen representations on its leaves, is
\emph{$(\rho,\kappa)$-admissible for $A$} if:
\begin{enumerate}[label=(A\arabic*)]
\item\label{A1} For every internal node $p$,
\[
\norm{pAp}\le\rho^{\lev(p)}\norm A.
\]
\item\label{A2} At each leaf $q$, the chosen factors $u_q,v_q\in q\M q$
satisfy
\[
qAq=\comm{u_q}{v_q},\qquad
\norm{u_q}\norm{v_q}\le\kappa\rho^{\lev(q)}\norm A.
\]
\end{enumerate}
For $qAq=0$ we take $u_q=v_q=0$.
\end{definition}

\noindent\textit{Example.}
Let $p_1,p_2,q\in\M$ be nonzero mutually orthogonal projections
with $p_1+p_2+q=1$, and put $p=p_1+p_2$.
First split $1=p+q$, then split $p=p_1+p_2$, leaving $q$ unsplit:
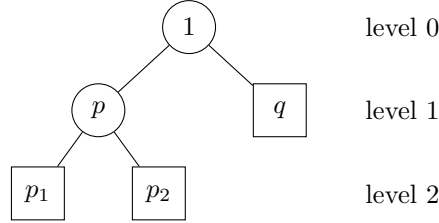
\begin{center}
\begin{tikzpicture}[x=1cm,y=1cm,
  every node/.style={font=\small},
  internal/.style={circle,draw,minimum size=7mm,inner sep=1pt},
  leaf/.style={rectangle,draw,minimum size=7mm,inner sep=2pt}]
\node[internal] (root) at (0,0) {$1$};
\node[internal] (p) at (-1.2,-1.1) {$p$};
\node[leaf] (q) at (1.2,-1.1) {$q$};
\node[leaf] (p1) at (-2,-2.2) {$p_1$};
\node[leaf] (p2) at (-0.4,-2.2) {$p_2$};
\draw (root) -- (p) (root) -- (q) (p) -- (p1) (p) -- (p2);
\node[anchor=west] at (2.2,0) {level $0$};
\node[anchor=west] at (2.2,-1.1) {level $1$};
\node[anchor=west] at (2.2,-2.2) {level $2$};
\end{tikzpicture}
\end{center}
Circles mark internal nodes and squares mark leaves. Here
$N_0=\{1\}$, $N_1=\{p\}$, $L_1=\{q\}$ and
$L_2=\{p_1,p_2\}$; all other $N_j,L_j$ are empty.
This is an assembly tree with branching bound $2$.
To make it admissible for $A$, one must also satisfy
\textup{(A1)} at $1,p$ and choose representations satisfying
\textup{(A2)} at $q,p_1,p_2$.

Any two nodes are orthogonal or comparable, and
$N_j\cup\bigcup_{i\le j}L_i$ is a finite partition of unity.
The next theorem assembles the leaf representations and the
interactions between sibling corners into one commutator, provided
the compression norms decay sufficiently rapidly.

\begin{theorem}\label{thm:tree}
Let $r\ge2$ be an integer, put
\[
g_r=\sqrt{r/2},\qquad
\gamma=\frac1{8g_r+2}=\frac1{4\sqrt{2r}+2},
\]
and let $0<\rho\le\gamma/2$, $\kappa\ge0$.  Suppose $A\in\M$
admits a $(\rho,\kappa)$-admissible assembly tree with branching bound $r$.
Then there exist $B,C\in\M$ with
\[
A=\comm BC,\qquad
\norm B\le\tfrac18,\qquad
\norm C\le\bigl(4\sqrt{2r}+2\bigr)
\Bigl(\tfrac83(r-1)+16\kappa\rho\Bigr)\norm A .
\]
In particular
$\norm B\norm C\le\bigl(3\,r^{3/2}+\kappa\bigr)\norm A$.
\end{theorem}

\begin{proof}[Proof of Theorem \ref{thm:tree}]

The case $A=0$ is immediate. Normalize $\norm A=1$ and put
$\beta=1/8$. Choose $r$ distinct points
$\zeta^{(1)},\ldots,\zeta^{(r)}$ from the
$\lceil\sqrt r\rceil\times\lceil\sqrt r\rceil$ unit grid centered
at zero. Their moduli are at most $g_r=\sqrt{r/2}$ and their
pairwise distances are at least $1$. We shall repeatedly use the identity

\begin{equation}\label{eq:magic}
(g_r+\beta)\,\frac\gamma{1-\gamma}
=\frac{g_r+\tfrac18}{8g_r+1}
=\frac18,
\qquad\text{hence}\qquad
\sum_{j\ge1}(g_r+\beta)\gamma^j=\frac18 .
\end{equation}

For each internal node $p$ with children
at level $j=\lev(p)+1$, assign to its $m(p)\le r$ children
distinct centers
$\zeta_c\in\gamma^j\cdot\{\zeta^{(1)},\dots,\zeta^{(r)}\}$ from
the chosen grid scaled by $\gamma^j$; thus for
siblings $i\ne s$ at level $j$,
\begin{equation}\label{eq:sep}
|\zeta_i|\le g_r\gamma^j,\qquad
|\zeta_i-\zeta_s|\ge\gamma^j .
\end{equation}
Set $z_1=0$ at the root and define accumulated centers
recursively by $z_c=z_p+\zeta_c$ for a child $c$ of $p$.

For each leaf $q$ at level $j$
with $qAq\ne0$, rescale its representation: put
\[
\widehat u_q=\theta_q u_q,\qquad
\widehat v_q=\theta_q^{-1}v_q,\qquad
\theta_q=\frac{\beta\gamma^{j}}{\norm{u_q}},
\]
so that $\comm{\widehat u_q}{\widehat v_q}=qAq$,
$\norm{\widehat u_q}=\beta\gamma^{j}$, and, by \ref{A2},
\begin{equation}\label{eq:vhat}
\norm{\widehat v_q}
=\frac{\norm{u_q}\norm{v_q}}{\beta\gamma^{j}}
\le8\kappa\Bigl(\frac\rho\gamma\Bigr)^{j} .
\end{equation}
For $qAq=0$ put $\widehat u_q=\widehat v_q=0$.  Define, for
$j\ge1$,
\[
B^{(j)}
=\sum_{c:\,\lev(c)=j}\zeta_c\,c
\ +\sum_{q\in L_j}\widehat u_q ,
\qquad
B=\sum_{j\ge1}B^{(j)} .
\]
Each $B^{(j)}$ is block-diagonal over the level-$j$ nodes, with
block norms $\le(g_r+\beta)\gamma^j$; by \eqref{eq:magic} the
series converges in norm and
$\norm B\le\tfrac18$.

Every node projection commutes with $B$: node projections are
nested or orthogonal, and each leaf term is supported on a leaf,
which has no descendants. This property passes to the norm limit.
We call $B$ \emph{tree-diagonal} and set
$B_c=cBc$ and $D_c=B_c-z_cc$ for each node $c$.

\noindent\textit{Claim.}
For every node $c$ at level $j\ge1$,
$\norm{D_c}\le\gamma^{j}/8$.

\noindent\textit{Proof of the claim.}
If $c$ is a leaf, $B_c=z_cc+\widehat u_c$, so
$\norm{D_c}=\norm{\widehat u_c}\le\beta\gamma^j=\gamma^j/8$.  If $c$
is internal, $D_c$ collects the level-$j'$ increments and
$\widehat u$'s inside $c$ for all $j'>j$; the level-$j'$
contribution is block-diagonal with block norms
$\le(g_r+\beta)\gamma^{j'}$, so
\[
\norm{D_c}\le\sum_{j'\ge j+1}(g_r+\beta)\gamma^{j'}
=(g_r+\beta)\frac{\gamma^{j+1}}{1-\gamma}
=\frac{\gamma^{j}}8
\]
by \eqref{eq:magic}.

Fix $j\ge1$ and an internal node
$p\in N_{j-1}$ with children $c_1,\dots,c_m$ at level $j$.  For
each ordered pair $i\ne s$ solve for $C_{p;i,s}\in c_i\M c_s$:
\begin{equation}\label{eq:syl-eq}
B_{c_i}C_{p;i,s}-C_{p;i,s}B_{c_s}=c_iAc_s .
\end{equation}
The sibling centers satisfy
$z_{c_i}-z_{c_s}=\zeta_{c_i}-\zeta_{c_s}$.
Writing $B_{c_i}=z_{c_i}c_i+D_{c_i}$, \eqref{eq:syl-eq} becomes
$(\zeta_{c_i}-\zeta_{c_s})C_{p;i,s}+D_{c_i}C_{p;i,s}-C_{p;i,s}D_{c_s}
=c_iAc_s$.  By \eqref{eq:sep} and the claim,
\[
|\zeta_{c_i}-\zeta_{c_s}|-\norm{D_{c_i}}-\norm{D_{c_s}}
\ \ge\ \gamma^j-2\cdot\frac{\gamma^j}8=\frac{3\gamma^j}4>0,
\]
so the Sylvester estimate \eqref{eq:sylvbound} provides a solution.
Using \ref{A1} and $\norm{c_iAc_s}\le\norm{pAp}$, we obtain
\begin{equation}\label{eq:Cis}
\norm{C_{p;i,s}}
\le\frac{\norm{pAp}}{(3/4)\gamma^{j}}
\le\frac43\,\frac{\rho^{j-1}}{\gamma^{j}} .
\end{equation}
Define
\[
C^{(j)}
=\sum_{p\in N_{j-1}}\ \sum_{\substack{1\le i,s\le m(p)\\i\ne s}}
C_{p;i,s}
\ +\ \sum_{q\in L_j}\widehat v_q,
\qquad
C=\sum_{j\ge1}C^{(j)} .
\]
Fix $p\in N_{j-1}$.  Inside $p$, the sibling terms form a hollow
block matrix over $\le r$ blocks, of norm
$\le(r-1)\cdot\frac43\rho^{j-1}\gamma^{-j}$ by
the hollow block estimate \eqref{eq:hollowbound} and \eqref{eq:Cis}; the leaf terms
$\widehat v_q$ are block-diagonal of norm
$\le8\kappa(\rho/\gamma)^j$ by \eqref{eq:vhat}.  Across the
different $p\in N_{j-1}$ the pieces are orthogonal blocks, so
\begin{equation}\label{eq:Cj}
\norm{C^{(j)}}
\le\Bigl(\frac43\,\frac{r-1}\rho+8\kappa\Bigr)
\Bigl(\frac\rho\gamma\Bigr)^{j} .
\end{equation}
Since $\rho/\gamma\le\tfrac12$, the series converges in norm and
\begin{equation}\label{eq:Cnorm}
\norm C
\le\Bigl(\frac43\,\frac{r-1}\rho+8\kappa\Bigr)
\frac{2\rho}\gamma
=\frac1\gamma\Bigl(\frac83(r-1)+16\kappa\rho\Bigr),
\end{equation}
which is the bound in the statement.

Since $B$ is tree-diagonal and
$C_{p;i,s}=c_iC_{p;i,s}c_s$,
\[
\comm B{C_{p;i,s}}
=Bc_iC_{p;i,s}-C_{p;i,s}c_sB
=B_{c_i}C_{p;i,s}-C_{p;i,s}B_{c_s}
=c_iAc_s
\]
by \eqref{eq:syl-eq}.  Likewise, for a leaf $q$ at level $j$,
$\widehat v_q\in q\M q$ and $B_q=z_qq+\widehat u_q$, so the scalar
drops out:
$\comm B{\widehat v_q}=\comm{\widehat u_q}{\widehat v_q}=qAq$.
Summing over the finitely many terms of $C^{(j)}$,
\begin{equation}\label{eq:levelj}
\comm B{C^{(j)}}
=\sum_{p\in N_{j-1}}
\Bigl(pAp-\sum_{c\text{ child of }p}cAc\Bigr)
+\sum_{q\in L_j}qAq .
\end{equation}
We claim that for every $J\ge0$,
\begin{equation}\label{eq:telescope}
\sum_{j=1}^{J}\comm B{C^{(j)}}
= A-\sum_{p\in N_J}pAp .
\end{equation}
For $J=0$, both sides vanish since $N_0=\{1\}$.  Assuming
\eqref{eq:telescope} for $J$, adding \eqref{eq:levelj} with
$j=J+1$ increases the left side by
$\sum_{p\in N_J}pAp-\sum_{c\in N_{J+1}}cAc$, because the children
of the nodes of $N_J$ are exactly $L_{J+1}\sqcup N_{J+1}$.  This
is the required change of the right side.

Finally, $\sum_{p\in N_J}pAp$ is block-diagonal, so by \ref{A1}
its norm is at most
$\rho^J\to0$.  Since
$C=\lim_J\sum_{j\le J}C^{(j)}$ in norm and multiplication is
norm-continuous on bounded sets,
\[
\comm BC=\lim_{J\to\infty}\sum_{j=1}^J\comm B{C^{(j)}}
=\lim_{J\to\infty}\Bigl(A-\sum_{p\in N_J}pAp\Bigr)=A .
\]

Finally, by the construction of $B$ and
\eqref{eq:Cnorm},
\[
\norm B\norm C
\le\frac18\bigl(4\sqrt{2r}+2\bigr)
\Bigl(\frac83(r-1)+16\kappa\rho\Bigr).
\]
Since $r\ge2$ and $\rho\le\gamma/2$,
\[
\frac13(4\sqrt{2r}+2)(r-1)
\le\left(\frac{4\sqrt2}{3}+\frac{2}{3\sqrt2}\right)r^{3/2}
\le3r^{3/2},\qquad
\frac{2\kappa\rho}{\gamma}\le\kappa.
\]
Thus $\norm B\norm C\le3r^{3/2}+\kappa$.
\end{proof}

\subsection{Proof of Theorem~\ref{thm:factor}}\label{sec:complete}

\begin{proof}
Set
\[
R=2^{12},\qquad s=2^{-12},\qquad t=2^{-19},\qquad
\rho=2^{-10}.
\]
Then $t\le s/\sqrt{2R}$, $\rho=2s+2/R$, and
\[
\rho\le\frac1{8\sqrt{2R}+4},
\]
because $8\sqrt{2R}+4<729<2^{10}$.  Thus
Theorem~\ref{thm:mass} applies to every low-$L^2$ element in every
corner, while Theorem~\ref{thm:high} applies to the complementary
high-$L^2$ case.

We may assume $A\ne0$.  If $\norm A_2\ge t\norm A$, the conclusion
follows directly from Theorem~\ref{thm:high}.  Otherwise construct an
assembly tree recursively.  At a node $p$ carrying the element
$X_p=pAp$ in the corner $p\M p$, use the normalized corner trace and
$L^2$-norm.  If $X_p=0$, make $p$ a leaf and use the zero
representation.  If $\norm{X_p}_2\ge t\norm{X_p}$, make $p$ a leaf
and use Theorem~\ref{thm:high}; the resulting representation has
product at most $K_H(t)\norm{pAp}$.  In the remaining case,
Theorem~\ref{thm:mass} supplies $R$ children whose diagonal corners
have trace zero and norm at most $\rho\norm{pAp}$.

The root is internal in this case.  Induction down the tree gives
\[
\norm{pAp}\le\rho^{\lev(p)}\norm A
\]
for every node, including the leaves: each child compression has
norm at most $\rho$ times that of its parent. At a leaf $q$,
the chosen representation therefore has norm product at most
$K_H(t)\rho^{\lev(q)}\norm A$. Both admissibility conditions
hold with $\kappa=K_H(t)$.  Theorem~\ref{thm:tree}
therefore gives
\[
A=\comm BC,\qquad
\norm B\norm C\le\bigl(3R^{3/2}+K_H(t)\bigr)\norm A
 =\bigl(3\cdot2^{18}+K_H(2^{-19})\bigr)\norm A.
\]
This is the asserted estimate.
\end{proof}

\section{Passage to arbitrary finite von Neumann algebras}
\label{sec:finite}

We complete the proof of Theorem~\ref{thm:main} by combining
Theorem~\ref{thm:factor} with the matrix theorem of
\cite[Theorem~1.1]{DIC}. Here $\M$ is a general finite von Neumann
algebra, and we use its normalized center-valued trace
\[
T_\M:\M\longrightarrow Z(\M).
\]
This is the faithful normal conditional expectation onto the
center satisfying
$T_\M(XY)=T_\M(YX)$. We use its standard properties, including
$Z(\M)$-linearity and contractivity; see \cite[Section~8.2]{KR}.
For a central projection $z$, its restriction to $z\M$ is the
normalized center-valued trace of that summand, with unit $z$.

We pass from finite factors to general finite algebras by selecting
uniformly bounded commutator factors in a central decomposition.
To remove the separability assumption, we construct a subalgebra
with separable predual in which the center-valued trace of the
given element still vanishes. The general case then follows by
assembling the factors over orthogonal central summands.

\subsection{Bounded measurable selection}

To assemble representations in a direct integral, the two factors
must be both measurable and essentially bounded. The following
lemma obtains such sections from uniformly bounded solutions in
the individual fibers.

\begin{lemma}\label{lem:measurablecomm}
Let $(X,\mu)$ be a standard probability space and
$(\M_x)_{x\in X}$ a measurable field of von Neumann algebras with
separable preduals. Let $x\mapsto A_x\in\M_x$ be a measurable
section, and let $L>0$. Suppose that, for almost every $x$, there
are $b,c\in\M_x$ satisfying
\[
\comm bc=A_x,\qquad\norm b,\norm c\le L.
\]
Then there are measurable sections $x\mapsto B_x,C_x\in\M_x$
satisfying these conditions almost everywhere. In particular,
these sections define elements of the direct integral algebra.
\end{lemma}

\begin{proof}
Equip the measurable field with the standard Borel structure of
\cite[Section~8]{VW}. After restricting
to a conull Borel subset of $X$, the total space
$\mathscr M=\bigsqcup_{x\in X}\M_x$ is standard Borel, its
projection $\pi:\mathscr M\to X$ is Borel, and the section $A$ is
Borel. In this structure addition, scalar multiplication, adjoint
and multiplication are Borel on their respective domains;
see \cite[Proposition~5.4, Definition~8.1 and
Propositions~8.4, 8.8]{VW}.
The norm is Borel as well. Indeed, a norm-dense sequence of Borel
sections $\omega_n(x)$ in the unit balls of the preduals gives
\[
\norm b=\sup_n|\omega_n(\pi(b))(b)|,
\]
and each evaluation on the right is Borel.

Consequently the fiber product
\[
\mathscr P_L=
\{(b,c)\in\mathscr M\times\mathscr M:
\pi(b)=\pi(c),\ \norm b,\norm c\le L\}
\]
is a standard Borel space, and its subset
\begin{equation}\label{eq:commutatorgraph}
\mathscr S=
\{(b,c)\in\mathscr P_L:bc-cb=A_{\pi(b)}\}
\end{equation}
is Borel, since the two sides of the defining equation are Borel
maps into the standard Borel space $\mathscr M$.
The projection of $\mathscr S$ onto $X$ is analytic and, by
hypothesis, conull. Restricting to a conull Borel subset, we may
therefore assume that every fiber is nonempty.

Apply the Jankov--von Neumann uniformization theorem
\cite[Theorem~18.1]{Kechris} to the Borel relation
\[
\{(x,(b,c))\in X\times\mathscr S:\pi(b)=x\}.
\]
It gives a section $x\mapsto(B_x,C_x)$ measurable for the
completion of $\mu$. Since the target is standard Borel, this
section has a Borel representative on a conull Borel subset.
The equation and norm bounds hold on this conull set. The
identification of the abstract and concrete Borel structures in
\cite[Proposition~8.4]{VW} shows that the selected sections are
measurable operator fields. Their essential boundedness gives
elements of the direct integral algebra with norms at most $L$.
\end{proof}

\subsection{Separable reduction}

We next remove the separability assumption for algebras admitting
a faithful normal tracial state. The required subalgebra must
contain the given element and preserve vanishing of its
center-valued trace.

\begin{lemma}\label{lem:separablereduction}
Suppose that $\M$ has a faithful normal tracial state and
$T_\M(A)=0$. There is a unital von Neumann subalgebra
$\N\subseteq\M$ with separable predual such that
$A\in\N$ and $T_\N(A)=0$.
\end{lemma}

\begin{proof}
The Dixmier averaging theorem for finite von Neumann algebras
\cite[Theorem~8.3.6]{KR} gives
\[
0=T_\M(A)\in
\overline{\operatorname{conv}}^{\norm\cdot}
\{uAu^*:u\in\mathcal U(\M)\}.
\]
For every $n\ge1$, choose unitaries $u_{n,j}$ and weights
$\lambda_{n,j}\ge0$ for $1\le j\le m_n$, with
$\sum_j\lambda_{n,j}=1$, such that
\begin{equation}\label{eq:separableaverages}
D_n:=\sum_{j=1}^{m_n}\lambda_{n,j}u_{n,j}Au_{n,j}^*,
\qquad\norm{D_n}<2^{-n}.
\end{equation}
Let $\mathcal A$ be the unital C$^*$-algebra generated by $A$ and
all the $u_{n,j}$, and let $\N=W^*(\mathcal A)\subseteq\M$.
The algebra $\mathcal A$ is norm separable. The given trace
restricts to a faithful normal tracial state $\tau$ on $\N$.
In its faithful normal left representation on $L^2(\N,\tau)$,
Kaplansky density shows that $\mathcal A$ is dense in $\N$ for
the $L^2$-norm: bounded strong approximations converge on the
vector $1\in L^2(\N,\tau)$. Since
$\norm x_2\le\norm x$, a countable norm-dense subset of
$\mathcal A$ is $L^2$-dense in $L^2(\N,\tau)$.
Thus this Hilbert space is separable, and the faithfully and
normally represented algebra $\N$ has separable predual.

Every unitary in \eqref{eq:separableaverages} belongs to $\N$.
Unitary invariance and contractivity of $T_\N$ therefore give
\[
T_\N(D_n)=T_\N(A),\qquad
\norm{T_\N(A)}\le\norm{D_n}<2^{-n}.
\]
It follows that $T_\N(A)=0$.
\end{proof}

Including the averaging unitaries ensures that $T_\N(A)=0$,
without requiring any relation between $Z(\N)$ and $Z(\M)$.
Taking $\N=W^*(A)$ alone would not suffice: for
$A=\operatorname{diag}(1,-1)\in M_2(\C)$, the algebra $W^*(A)$
is abelian and its center-valued trace fixes $A$.

\subsection{Proof of Theorem~\ref{thm:main}}

\begin{proof}

Assume $T_\M(A)=0$ and $A\ne0$; for $A=0$, take $B=C=0$.
Set $K=\max\{K_{\mathrm I},K'\}$ as in \eqref{eq:finiteconstant}.

We first balance the factors in the two
factor theorems. If $a=\comm uv\ne0$, put
\[
\alpha=\sqrt{\norm v/\norm u},\qquad
b=\alpha u,\quad c=\alpha^{-1}v.
\]
Then $a=\comm bc$ and
$\norm b=\norm c=\sqrt{\norm u\norm v}$.
For $a=0$, take $b=c=0$. Every finite factor is either a matrix
algebra or of type $\mathrm{II}_1$, so the two factor theorems
and \eqref{eq:finiteconstant} give
\begin{equation}\label{eq:balancedfinitefactor}
\tau(a)=0\quad\Longrightarrow\quad
 a=\comm bc,\qquad
\norm b,\norm c\le\sqrt{K\norm a}
\end{equation}
in every finite factor. In the scalar factor the hypothesis
forces $a=0$.

Assume first that $\M_*$ is separable. The central decomposition theorem
for finite von Neumann algebras gives
\begin{equation}\label{eq:centraldecomposition}
(\M,\tau)=\int_X^\oplus(\M_x,\tau_x)\,d\mu(x),
\qquad
Z(\M)=L^\infty(X,\mu),
\end{equation}
where $(X,\mu)$ is a standard probability space and almost every
$(\M_x,\tau_x)$ is a finite factor with its normalized trace;
see \cite[Chapter~IV, Section~8]{Takesaki}.
Here $\tau$ is a faithful normal tracial state of $\M$,
and we use the direct-integral notation introduced in
Section~\ref{sec:prelim}.
Under this decomposition,
\begin{equation}\label{eq:fibertrace}
T_\M(A)_x=\tau_x(A_x)1_x\quad\text{almost everywhere}.
\end{equation}
Indeed, the field on the right is central and bounded by
$\norm A$. For $h\in L^\infty(X,\mu)$, disintegration of the trace
gives
\[
\tau(hA)=\int_X h(x)\tau_x(A_x)\,d\mu(x).
\]
Thus the map in \eqref{eq:fibertrace} is the $\tau$-preserving
conditional expectation onto the center. The center-valued trace
$T_\M$ also preserves $\tau$: Dixmier averaging approximates
$T_\M(X)$ in norm by convex combinations of unitary conjugates
of $X$, all having trace $\tau(X)$. Uniqueness of the
$\tau$-preserving conditional expectation proves
\eqref{eq:fibertrace}.

Choose a measurable representative of $A$ with
$\norm{A_x}\le\norm A$ almost everywhere.
By \eqref{eq:fibertrace}, $\tau_x(A_x)=0$ almost everywhere.
Equation~\eqref{eq:balancedfinitefactor} shows that the hypotheses
of Lemma~\ref{lem:measurablecomm} hold with
$L=\sqrt{K\norm A}$. It supplies measurable
sections $B_x,C_x$ with $\comm{B_x}{C_x}=A_x$ and
$\norm{B_x},\norm{C_x}\le L$ almost everywhere.
Their direct integrals $B,C\in\M$ satisfy $A=\comm BC$, because
multiplication of decomposable bounded operators is fiberwise,
and
\[
\norm B=\mathop{\mathrm{ess\,sup}}_x\norm{B_x}\le L,
\qquad
\norm C\le L.
\]
This gives the required commutator representation in this case.

Now suppose only that $\M$ admits a faithful normal tracial
state. Lemma~\ref{lem:separablereduction} gives a unital
subalgebra $\N\subseteq\M$ with separable predual such that
$A\in\N$ and $T_\N(A)=0$. The result just proved applies in
$\N$, yielding a commutator representation in $\M$ with the
same bounds on the two factors, since the inclusion is isometric.

For a general finite $\M$, put $Z=Z(\M)$ and choose a maximal
orthogonal family of nonzero
projections $(z_i)_{i\in I}$ in $Z$ such that each $z_iZ$ admits
a faithful normal state $\varphi_i$. Such a family exists by
Zorn's lemma. Its supremum is $1$: otherwise the nonzero
remaining central projection supports a nonzero normal positive
functional on $Z$. The support projection $z$ of that functional
is nonzero, orthogonal to all $z_i$, and the normalized restriction
of the functional to $zZ$ is a faithful normal state, contradicting
maximality.

For $\M_i=z_i\M$, define
\[
\tau_i(X)=\varphi_i(T_\M(X))\qquad(X\in\M_i).
\]
This is a normal tracial state with $\tau_i(z_i)=1$. It is
faithful: if $X\ge0$ and $\tau_i(X)=0$, then faithfulness of
$\varphi_i$ on $z_iZ$ gives $T_\M(X)=0$, and faithfulness of
$T_\M$ gives $X=0$. Moreover,
$T_{\M_i}(z_iA)=z_iT_\M(A)=0$.
The result for algebras admitting a faithful normal tracial state
therefore supplies $B_i,C_i\in\M_i$ such that
\[
z_iA=\comm{B_i}{C_i},\qquad
\norm{B_i},\norm{C_i}
\le\sqrt{K\norm{z_iA}}
\le\sqrt{K\norm A}.
\]
Represent $\M$ on a Hilbert space. The nets of finite sums of
the $B_i$ and of the $C_i$ converge strongly: by orthogonality,
for every vector $\xi$ and finite $F\subset I$,
\[
\Big\|\sum_{i\in F}B_i\xi\Big\|^2
=\sum_{i\in F}\norm{B_i z_i\xi}^2
\le K\norm A\sum_{i\in F}\norm{z_i\xi}^2,
\]
and the same bound applied to tails makes the net Cauchy.
The strong limits $B,C$ lie in $\M$ and have norms at most
$\sqrt{K\norm A}$. Since $z_i$ is central,
$z_i\comm BC=\comm{B_i}{C_i}=z_iA$ for every $i$.

Since $\sum_i z_i=1$ strongly, it follows that $\comm BC=A$.
The bounds on $B$ and $C$ give $\norm B\norm C\le K\norm A$,
completing the proof.
\end{proof}

\section{An explicit bound}\label{sec:conclusion}
Theorem~\ref{thm:main} admits the unoptimized numerical choice
\begin{equation}\label{eq:numericalK}
K=2^{\,2^{\,2^{46}}}.
\end{equation}

For $x\ge1$, \eqref{eq:f} gives
\[
\growthF(x)\le900x^9.
\]
At $t=2^{-19}$, put $L=2^{44}$ and
$y=\growthF^{\circ L}(\sqrt5\,2^{19})$. Iteration yields
\[
\log_2 y
\le9^L\log_2(\sqrt5\,2^{19})
  +\frac{9^L-1}{8}\log_2 900
<23\cdot9^L.
\]
Since $2\Lambda_0<2^{21}$, the formula in
Proposition~\ref{prop:matrixcomm} gives, for $x\ge1$,
\[
2(2\Lambda_0x)+x+2\le2^{23}x,\qquad
\sigma(x,2\Lambda_0x)\le2^{48}x^2,\qquad
\Pi^*(x)\le2^{220}x^9.
\]
The factors $5/t^2$ and $a_1^{-2}$ cancel in \eqref{eq:KH}.
Using $L+1\le2^{45}$, we therefore obtain
\[
K_H(2^{-19})
=y^2(L+1)\bigl(4\Pi^*(y)+2(L+1)y\bigr)
\le2^{270}y^{11}.
\]
Consequently, the choice $K'=3\cdot2^{18}+K_H(2^{-19})$ satisfies
\[
\log_2 K'<271+253\cdot9^L<2^{10}9^L,
\qquad
\log_2\log_2 K'<10+L\log_2 9<2^{46}.
\]
The matrix estimate in \cite{DIC} admits a choice
$\log_2\log_2 K_{\mathrm I}<2^{31}$: its analogous iteration has
$2^{29}$ steps and initial value $2^{27}$.
Together with \eqref{eq:finiteconstant}, these inequalities justify
the numerical choice \eqref{eq:numericalK}.

The dominant growth comes from the repeated similarities in the
block-elimination argument. Better control of these similarities
could substantially reduce the bound.

\section*{Acknowledgements}
I thank Wei Yuan for drawing my attention to the results of
Wen, Fang and Yao used in this paper. I am grateful to Sizhuo
Yan for helpful explanations of these results and for sharing
useful ideas. I would also like to thank Prof. Lihong Zhi for
valuable comments and suggestions on the manuscript.

\section*{Declaration of generative AI}
GPT-6 Astra and GPT-5.6 Sol were used during this work to assist
in exploring proof strategies and revising the exposition.
The author assumes full responsibility for all proofs and
mathematical details and for the mathematical correctness of
the final results.



\begin{thebibliography}{12}
\expandafter\ifx\csname natexlab\endcsname\relax\def\natexlab#1{#1}\fi
\providecommand{\DOIprefix}{doi:}
\providecommand{\ArXivprefix}{arXiv:}
\providecommand{\URLprefix}{URL: }
\providecommand{\Pubmedprefix}{pmid:}
\providecommand{\bibinfo}[2]{#2}
\ifx\xfnm\relax \def\xfnm[#1]{\unskip,\space#1}\fi
\bibitem[{Anantharaman and Popa()}]{AP}
\bibinfo{author}{Anantharaman, C.}, \bibinfo{author}{Popa, S.}
\newblock \bibinfo{title}{An introduction to {$\mathrm{II}_1$} factors}.
\newblock \bibinfo{note}{Draft lecture notes}.
\newblock \URLprefix \url{https://www.math.ucla.edu/~popa/Books/IIun.pdf}.

\bibitem[{Cao et~al.(2024)Cao, Fang and Yao}]{CFY}
\bibinfo{author}{Cao, X.}, \bibinfo{author}{Fang, J.},
  \bibinfo{author}{Yao, Z.}, \bibinfo{year}{2024}.
\newblock \bibinfo{title}{On finite sums of projections and {Dixmier}'s
  averaging theorem for type {$\mathrm{II}_1$} factors}.
\newblock \bibinfo{journal}{Journal of Functional Analysis}
  \bibinfo{volume}{287}(8), \bibinfo{pages}{110568}.
\newblock \DOIprefix\href{https://doi.org/10.1016/j.jfa.2024.110568}{\nolinkurl{10.1016/j.jfa.2024.110568}}.

\bibitem[{Fack and de~la Harpe(1980)}]{FdH}
\bibinfo{author}{Fack, T.}, \bibinfo{author}{de~la Harpe, P.},
  \bibinfo{year}{1980}.
\newblock \bibinfo{title}{Sommes de commutateurs dans les alg\`ebres de von
  {Neumann} finies continues}.
\newblock \bibinfo{journal}{Annales de l'Institut Fourier}
  \bibinfo{volume}{30}(3), \bibinfo{pages}{49--73}.
\newblock \DOIprefix\href{https://doi.org/10.5802/aif.792}{\nolinkurl{10.5802/aif.792}}.
\bibitem[{Johnson et~al.(2013)Johnson, Ozawa and Schechtman}]{JOS}
\bibinfo{author}{Johnson, W.B.}, \bibinfo{author}{Ozawa, N.},
  \bibinfo{author}{Schechtman, G.}, \bibinfo{year}{2013}.
\newblock \bibinfo{title}{A quantitative version of the commutator theorem for
  zero trace matrices}.
\newblock \bibinfo{journal}{Proceedings of the National Academy of Sciences}
  \bibinfo{volume}{110}(48), \bibinfo{pages}{19251--19255}.
\newblock \DOIprefix\href{https://doi.org/10.1073/pnas.1202411109}{\nolinkurl{10.1073/pnas.1202411109}}.
\bibitem[{Kadison and Ringrose(1986)}]{KR}
\bibinfo{author}{Kadison, R.V.}, \bibinfo{author}{Ringrose, J.R.},
  \bibinfo{year}{1986}.
\newblock \bibinfo{title}{Fundamentals of the Theory of Operator Algebras}.
  Volume \bibinfo{volume}{II: Advanced Theory}.
\newblock \bibinfo{publisher}{Academic Press}, \bibinfo{address}{Orlando, FL}.
\bibitem[{Kechris(1995)}]{Kechris}
\bibinfo{author}{Kechris, A.S.}, \bibinfo{year}{1995}.
\newblock \bibinfo{title}{Classical Descriptive Set Theory}. Volume
  \bibinfo{volume}{156} of \textit{\bibinfo{series}{Graduate Texts in
  Mathematics}}.
\newblock \bibinfo{publisher}{Springer-Verlag}, \bibinfo{address}{New York}.
\bibitem[{Marcoux(2006)}]{Marcoux}
\bibinfo{author}{Marcoux, L.W.}, \bibinfo{year}{2006}.
\newblock \bibinfo{title}{Sums of small number of commutators}.
\newblock \bibinfo{journal}{Journal of Operator Theory} \bibinfo{volume}{56}(1),
  \bibinfo{pages}{111--142}.
\newblock \URLprefix \url{https://jot.theta.ro/jot/archive/2006-056-001/2006-056-001-006.html}.
\bibitem[{Shen et~al.(2026)Shen, Wang and Zhi}]{DIC}
\bibinfo{author}{Shen, H.}, \bibinfo{author}{Wang, J.}, \bibinfo{author}{Zhi,
  L.}, \bibinfo{year}{2026}.
\newblock \bibinfo{title}{A dimension-independent commutator bound}.
\newblock \URLprefix \url{https://arxiv.org/abs/2609.09938}.
  \bibinfo{note}{arXiv preprint arXiv:2609.09938}.
\bibitem[{Shoda(1936)}]{Shoda}
\bibinfo{author}{Shoda, K.}, \bibinfo{year}{1936}.
\newblock \bibinfo{title}{Einige {S}\"atze \"uber {M}atrizen}.
\newblock \bibinfo{journal}{Japanese Journal of Mathematics}
  \bibinfo{volume}{13}, \bibinfo{pages}{361--365}.
\newblock \DOIprefix\href{https://doi.org/10.4099/jjm1924.13.0_361}{\nolinkurl{10.4099/jjm1924.13.0_361}}.
\bibitem[{Takesaki(1979)}]{Takesaki}
\bibinfo{author}{Takesaki, M.}, \bibinfo{year}{1979}.
\newblock \bibinfo{title}{Theory of Operator Algebras {I}}.
\newblock \bibinfo{publisher}{Springer-Verlag}, \bibinfo{address}{New York}.
\bibitem[{Vaes and Wouters(2025)}]{VW}
\bibinfo{author}{Vaes, S.}, \bibinfo{author}{Wouters, L.},
  \bibinfo{year}{2025}.
\newblock \bibinfo{title}{Borel fields and measured fields of {Polish} spaces,
  {Banach} spaces, von {Neumann} algebras, and {$C^*$}-algebras}.
\newblock \bibinfo{journal}{Journal of the London Mathematical Society}
  \bibinfo{volume}{111}(4), \bibinfo{pages}{e70159}.
\newblock \DOIprefix\href{https://doi.org/10.1112/jlms.70159}{\nolinkurl{10.1112/jlms.70159}}.
\bibitem[{Wen et~al.(2024)Wen, Fang and Yao}]{WFY}
\bibinfo{author}{Wen, S.}, \bibinfo{author}{Fang, J.}, \bibinfo{author}{Yao,
  Z.}, \bibinfo{year}{2024}.
\newblock \bibinfo{title}{A stronger version of {Dixmier}'s averaging theorem
  and some applications}.
\newblock \bibinfo{journal}{Journal of Functional Analysis}
  \bibinfo{volume}{287}(8), \bibinfo{pages}{110569}.
\newblock \DOIprefix\href{https://doi.org/10.1016/j.jfa.2024.110569}{\nolinkurl{10.1016/j.jfa.2024.110569}}.

\end{thebibliography}
\end{document}